\documentclass{amsart}
\usepackage{dsfont}
\usepackage[utf8]{inputenc}
\usepackage[T1]{fontenc}
\usepackage{ tipa }
\usepackage{helvet}
\usepackage{color,comment}
\usepackage{mathrsfs}  
\usepackage{stmaryrd}  
\usepackage{amsmath,amsxtra,amsthm,amssymb,xr,enumerate,fullpage}
\usepackage[all]{xy}
\usepackage[bbgreekl]{mathbbol}
\usepackage{bbm}
\usepackage[english,french]{babel}
\DeclareMathSymbol{\invques}{\mathord}{operators}{`>}
\DeclareUnicodeCharacter{00BF}{\tmquestiondown}
\DeclareRobustCommand{\tmquestiondown}{%
  \ifmmode\invques\else\textquestiondown\fi
}

\newtheorem{theorem}{Theorem}[section]
\newtheorem{lemma}[theorem]{Lemma}

\newtheorem{corollary}[theorem]{Corollary}
\newtheorem{defn}[theorem]{Definition}

\newtheorem{remark}[theorem]{Remark}

\newtheorem{convention}[theorem]{Convention}

\newcommand{\Gal}{\operatorname{Gal}}

\newcommand{\BB}{\mathbb{B}}
\newcommand{\NN}{\mathbb{N}}
\newcommand{\QQ}{\mathbb{Q}}
\newcommand{\Qp}{\mathbb{Q}_p}
\newcommand{\Zp}{\mathbb{Z}_p}
\newcommand{\ZZ}{\mathbb{Z}}

\newcommand{\g}{\mathbf{g}}

\newcommand{\fp}{\mathfrak{p}}

\newcommand{\vp}{\varphi}

\newcommand{\cH}{\mathcal{H}}
\newcommand{\cO}{\mathcal{O}}

\newcommand{\Brig}{\BB_{{\rm rig},\Qp}^+}

\newcommand{\cyc}{\textup{cyc}}

\newcommand{\ff}{\mathfrak{f}}
\newcommand{\fL}{\mathfrak{L}}
\newcommand{\fm}{\mathfrak{M}}

\newcommand{\ac}{\textup{ac}}
\newcommand{\LL}{\Lambda}

\newcommand{\f}{\textup{\bf f}}

\newcommand{\lra}{\longrightarrow}
\newcommand{\ra}{\rightarrow}

\newcommand{\Bf}{\mathbf{f}}

\newcommand{\Tw}{\mathrm{Tw}}

\newcommand{\p}{\mathfrak{p}}
\newcommand{\q}{\mathfrak{q}}

\newcommand{\tlog}{\widetilde{\log}}

\usepackage{hyperref}

\begin{document}

\title{Functional Equation for $p$-adic Rankin-Selberg $L$-functions}

\author{K\^az\i m B\"uy\"ukboduk}
\address{K\^az\i m B\"uy\"ukboduk\newline UCD School of Mathematics and Statistics\\ University College Dublin\\ Ireland}
\email{kazim.buyukboduk@ucd.ie}

\author{Antonio Lei}
\address{Antonio Lei\newline
D\'epartement de Math\'ematiques et de Statistique\\
Universit\'e Laval\\
1045 Avenue de la M\'edecine\\
Qu\'ebec, QC\\
Canada G1V 0A6}
\email{antonio.lei@mat.ulaval.ca}

\thanks{The first author is partially supported by the European Commission Global Fellowship CriticalGZ. The second author is supported by the NSERC Discovery Grants Program 05710.}
\selectlanguage{english}
\begin{abstract} 
   We prove a functional equation for the three-variable $p$-adic $L$-function attached to the Rankin-Selberg convolution of a Coleman family and a CM Hida family, which was studied by Loeffler and Zerbes. Consequently, we deduce that an anticyclotomic $p$-adic $L$-function attached to a $p$-non-ordinary modular form vanishes identically in the indefinite setting. This is a crucial step towards the Iwasawa main conjecture for non-ordinary modular forms over the anticyclotomic $\Zp$-extension of an imaginary quadratic field in the indefinite setting.
\end{abstract}

\maketitle

\selectlanguage{english}
\section{Introduction}
In \cite{LZ1}, Loeffler and Zerbes constructed a three-variable $p$-adic Rankin-Selberg $L$-function attached to two families of modular forms, which is characterized by its interpolating property at the crystalline points of a three-parameter family (afforded by a pair of Coleman families and the cyclotomic variation). In the special case where one of the two families is ordinary, Loeffler \cite{loefflerERL} extended this work to prove an interpolation formula also at non-crystalline points.

The goal of this article is to prove a functional equation for the Loeffler-Zerbes three-variable $p$-adic Rankin-Selberg $L$-functions associated to the Rankin-Selberg product of a Coleman family and a CM Hida family. See Theorem~\ref{thm:mainfunctionaleqninthreevariables} for the precise formulation of this functional equation. The main ingredients of our proof are Loeffler's interpolation formula in \cite{loefflerERL}, the functional equation for the complex Rankin-Selberg $L$-function of Li \cite{WinnieLi79} and  an analysis of root numbers, which allow us to interpolate various ``fudge factors'' arising from the complex functional equation in $p$-adic families (see Theorem~\ref{thm:mainfunctionaleqninthreevariables} below).

 The $p$-adic $L$-function we study here plays an important role in \cite{BFsuper}, where we study the Iwasawa Theory of the Rankin-Selberg convolutions $f\otimes\chi$ (of the base change of a $p$-non-ordinary modular form $f$ to an imaginary quadratic field $K$ where $p$ splits, with a ray class character $\chi$). For example,  the Loeffler-Zerbes $p$-adic Rankin-Selberg $L$-function gives rise to a two-variable $p$-adic $L$-function for $f\otimes \chi$ over the $\Zp^2$-extension of $K$, which in turn allows one to formulate an Iwasawa main conjecture. Furthermore, on specializing to the anticyclotomic $\Zp$-extension of $K$, we may study the behaviour of the arithmetic invariants associated to the Rankin-Selberg convolution $f\otimes \chi$ (where $\chi$ is now taken as a ring class character) along this tower.  When the root number $\epsilon(f/K)$ is $-1$, we show  in Corollary~\ref{cor:vanishingoftthefullanticyclotomicpadicLfunction} that our functional equation implies that the anticyclotomic specialization of the  said $p$-adic $L$-function is identically zero (as a matter of fact, unless $k=2$, deforming along Coleman family seems necessary to achieve all this; see Remark~\ref{rem:noeasywayout} below). This is one of the main ingredients in the portion of our work~\cite{BFsuper} where we prove results towards indefinite $p$-non-ordinary anticyclotomic main conjectures for $f\otimes\chi$.

One of the main results in \cite{BFsuper} is that the 2-variable $p$-adic function of Loeffler-Zerbes can be decomposed into integral signed $p$-adic $L$-functions using certain logarithmic matrix coming from the theory of Wach modules. Since these $p$-adic $L$-functions have bounded denominators, they are more suitable for the study of Iwasawa of $f\otimes\chi$ over the $\Zp^2$-extension of $K$ than their unbounded counterparts. There is a reformulation of Iwasawa Main Conjectures in this set up involving these $p$-adic $L$-functions (c.f., \cite[Conjecture 4.15]{BFsuper}), which one may attack using the integral (signed) Beilinson-Flach Euler systems. In the special case where $a_p(f)=0$, we show in  Theorem~\ref{thm:L5} that the anticyclotomic specialization of some of these signed $p$-adic $L$-functions also vanish identically. In this portion, we follow the line of argument Castella and Wan given in \cite{castellawan1607}, where they have also proved a similar result for the signed $p$-adic $L$-functions associated to an elliptic curve (granted the functional equation for analytic $p$-adic $L$-functions).

\section{Set up}
Fix forever a prime $p\geq 5$ and an imaginary quadratic field $K$ where $(p)=\fp\fp^c$ splits. The superscript $c$ will always stand for the action of a fixed complex conjugation. We fix a modulus $\mathfrak{f}$ coprime to $p$ with the property that the ray class number of $K$ modulo $\mathfrak{f}$ is not divisible by $p$. We also fix once and for all  embeddings $\iota_\infty:\overline{\QQ}\hookrightarrow \mathbb{C}$ and $\iota_p: \overline{\QQ}\hookrightarrow \mathbb{C}_p$ as well as an isomorphism $j: \mathbb{C}\stackrel{\sim}{\lra}\mathbb{C}_p$ such the diagram
$$\xymatrix @C=0.3cm@R=0.2cm{&\mathbb{C}\ar[dd]^j\\ \overline{\QQ}\ar[ur]^{\iota_\infty}\ar[dr]_{\iota_p}&\\
&\mathbb{C}_p\\}$$
commutes. and suppose that the prime $\fp$ of $K$ lands inside the maximal ideal of $\cO_{\mathbb{C}_p}$.  Throughout the article, $\chi$ denotes    a fixed ray class character  modulo $\ff $ with $\chi(\fp)\neq \chi(\fp^c)$.

Let $K_\infty$ denote the  $\ZZ_p^2$-extension of $K$ with $\Gamma:=\Gal(K_\infty/K)\cong \ZZ_p^2$. We let $K_\cyc/K$ and $K_\ac/K$ denote the cyclotomic and the anticyclotomic $\Zp$-extensions of $K$  contained in $K_\infty$ respectively. We write $\Gamma_{\ac}:=\Gal(K_\ac/K)$ and $\Gamma_\cyc:=\Gal(K_\cyc/K)$.  We also set $\Gamma_\cyc^\circ:=\Gal(K(\mu_{p^\infty})/K)\cong \ZZ_p^\times$. Let $\Delta:=\Gal(K(\mu_p)/K)$ so that $\Gamma_\cyc^\circ=\Gamma_\cyc\times\Delta$. For $\q=\p,\p^c$, we let  $\Gamma_\q$ denote the Galois group of the maximal pro-$p$ extension of $K$ unramified outside $\q$. In particular, we have the decompositions
\[
\Gamma\cong \Gamma_\cyc\times \Gamma_\ac\cong \Gamma_\p\times\Gamma_{\p^c}.
\]
Note that $\Gamma_\cyc$ and $\Gamma_\ac$ are the eigenspaces of $\Gamma$ under  the complex conjugation, whereas $\Gamma_\p$ and $\Gamma_{\p^c}$ are interchanged by the complex conjugation.

For a finite flat extension $\cO$ of $\ZZ_p$ and for $\Gamma_?\in\{\Gamma,\Gamma_\ac,\Gamma_\cyc,\Gamma_\p,\Gamma_{\p^c}, \Gamma_\cyc^\circ\}$, we define the Iwasawa algebra  
$\LL_{\cO}(\Gamma_?):=\cO[[\Gamma_?]]$ with coefficients in $\cO$. If $L$ is the field of fractions of $\cO$, we write $\LL_L(\Gamma_?)$ for $\LL_{\cO}(\Gamma_?)\otimes L$.  
For $?\in\{\ac, \cyc,\p,\p^c\}$, we set
$$ \Lambda_L^{\dagger}(\Gamma_?):=\left\{\sum c_n(\gamma_?-1)^n\in L[[\gamma_?-1]]\,:\ \displaystyle{\lim_{n\ra \infty}}\,\, |c_n|r^n=0\,\, \forall r\in [0,1)\right\}$$
(where $\gamma_?\in \Gamma_?$ is an arbitrary topological generator) and 
$$ \Lambda_L^{\dagger}(\Gamma_\cyc^\circ):= \Lambda_L^{\dagger}(\Gamma_\cyc)\otimes_L L[\Delta].$$
For $r \in \mathbb{R}_{\geq 0}$ and we shall denote the set of Amice transforms of the $L$-valued $r$-tempered distributions on $\Gamma_?$ by $\cH_r(\Gamma_?)$; in particular, $\cH_0(\Gamma_?)=\LL_L(\Gamma_?)$.  For non-negative real numbers $u,v\in \mathbb{R}$, we define the ring $\cH_{u,v}(\Gamma):=\cH_u(\Gamma_\p)\widehat{\otimes} \cH_v(\Gamma_{\p^c})$, which can be identified with a subring of $\Lambda_L^\dagger(\Gamma)$.  Throughout, we assume that $\gamma_\cyc=\gamma_\p\gamma_{\p^c}$ and $\gamma_\ac=\gamma_\p\gamma_{\p^c}^{-1}$.

Let $f \in S_k(\Gamma_0(N_f))$ be a normalized cuspidal eigen-newform of level $N_f$, even weight $k\ge2$ and trivial nebentypus. Throughout, we assume that $p\nmid N_f$. Let $\alpha$ and $\beta$ be the two roots of the Hecke polynomial $X^2-a_p(f)X+p^{k-1}$. We assume that $\alpha\neq \beta$ and let $f^\alpha$ and $f^\beta$ denote the two $p$-stabilizations of $f$. 

From now on, we fix $L$ to be a finite extension of $\Qp$ inside $\mathbb{C}_p$ that contains the images of all Fourier coefficients of $f$ as well as $\alpha$ and $\beta$ under $\iota_p$. Furthermore, we assume that  $L$ contains the values of our fixed ray class character $\chi$. Let $v_p$ be the $p$-adic valuation on $L$ that is normalized by $v_p(p)=1$. We  set $s_\alpha=v_p(\alpha)$ and $s_\beta=v_p(\beta)$. When $a_p(f)=0$,  we have $s_\alpha=s_{\beta}=\frac{k-1}{2}$. For a   ring class character $\eta$ of $K$ whose conductor is prime to $N_f$, we write $\epsilon(f/K) =\pm1$ for the global root number for the Rankin-Selberg $L$-series $L(f/K\otimes\eta,s)$. As the notation suggests, this quantity is  independent of the choice of $\eta$.

\subsection{Algebraic Hecke characters of $K$}
For $\q=\fp,\fp^c$, we let $\Omega_\q$ denote the maximal pro-$p$ quotient of the ray class group modulo $\q^\infty$. The geometrically normalized Artin map $\frak{A}$ induces  identifications 
$$\frak{A}:\,\Omega_\q\stackrel{\sim}{\lra}{\Gamma_\q}\,,\,\Omega:=\Omega_\p\times\Omega_{\p^c}\stackrel{\sim}{\lra} \Gamma\,.$$  
We let $\Omega_\cyc,\Omega_\cyc^\circ$ and $\Omega_\ac$  denote the subgroups of $\Omega$ corresponding to $\Gamma_\cyc,\Gamma_\cyc^\circ$ and $\Gamma_\ac$  under $\frak{A}$.

The id\`elic description of $\Omega_\q$ gives rise to a surjective map $\cO_\q^\times \twoheadrightarrow \Omega_\q$. Since we have assumed that $p$ does not divide the class number of $K$, this map allows us to identify the $1$-units $U_\q\subset \cO_\q^\times$ with $\Omega_\q$. Given an element $y \in \cO_\q^\times$, we write $y=\langle y \rangle [y]$ where $\langle y \rangle$ is a $1$-unit and $[y] \in \cO_\q^\times$ is a root of unity congruent to $y$ modulo $\q$.
 
\begin{defn}
\item[i)]  The \emph{$p$-adic avatar} $\widehat{\Xi}$ of an algebraic Hecke character $\Xi$ is defined by setting
$$\widehat{\Xi}(x):=x_\p^a x_{\p^c}^b\,j(\Xi(x_{\rm fin}))=x_\p^a x_{\p^c}^b\,j(\Xi(x))j(\Xi(x_\infty))^{-1}\,.$$
\item[ii)] The $($$p$-adic$)$ Galois character of an algebraic Hecke character $\Xi$ is given as the composite map
$$G_K^{\rm ab}\stackrel{\frak{A}^{-1}}{\lra}\mathbb{A}_K^\times/K^\times \stackrel{\widehat{\Xi}}{\lra}\mathbb{C}_p^\times\,.$$
\item[iii)] Let $\Sigma$ denote the set of algebraic Hecke characters $\Xi$ whose associated Galois characters factor through $\Gamma$.
For each positive integer $\kappa$, we let $\Sigma_{+}(\kappa)\subset\Sigma$ be the subset of characters  of $\infty$-type $(a,b)$ with $1-k/2\leq a\le b\leq \kappa-k/2-1$ and set ${\Sigma_{+}=\cup_{\kappa\in \ZZ^+}\,\Sigma_{+}(\kappa)}$.
\item[iv)] Let $\widetilde{\Sigma}\supset {\Sigma}$ denote the set of algebraic Hecke characters $\Xi$ of conductor dividing $p^\infty$ and whose $p$-adic avatars factor through $\Omega_\p\times\Omega_\cyc^\circ$. We similarly define the subsets $\widetilde{\Sigma}_{+}(\kappa)\subset\widetilde{\Sigma}_{+}\subset \widetilde{\Sigma}$.
\end{defn}
\begin{convention}
If there is no fear of confusion, we shall denote both the $p$-adic avatar and the $p$-adic Galois character attached to a Hecke character $\Xi$ also by $\Xi$.
\end{convention}
\begin{remark}\label{remark_padic_Hecke_explicit}
Suppose $\Xi \in \widetilde\Sigma$. Using the fact that the class number of $K$ is prime to $p$, one may describe its $p$-adic avatar $\widehat{\Xi}$ explicitly as follows: The $p$-adic Hecke character $\widehat{\Xi}$ factors through a quotient $\mathbb{A}_{K}^\times/Y$ $($where $Y\supset K^\times\prod_{v\nmid p}\cO_v^\times$, with the usual convention for $\cO_v^\times$ at archimedean places$)$ such that the natural map 
$$j_p: \cO_{{\p}}^\times\times\cO_{{\p^c}}^\times\lra \mathbb{A}_{K}^\times/Y$$ 
is surjective. Suppose $x \mod Y$ is the image of $(u_\p,u_{\p^c})$ under $j_p$. Let us write $\Xi_{\p}$ and $\Xi_{\p^c}$ for the restriction of the local characters at $\p$ and $\p^c$ to $\cO_{{\p}}^\times$ and $\cO_{{\p^c}}^\times$, respectively. Then 
$$\widehat{\Xi}(x)=u_{\p}^a u_{\p^c}^b j(\Xi_{\p}(u_{\p}))j(\Xi_{\p^c}(u_{\p^c})).$$ 
\end{remark}

\begin{lemma}
\label{lem:rigidityofheckechars}
Suppose $\Xi \in \widetilde{\Sigma}$ is an algebraic Hecke character of $\infty$-type $(a,b)$. Then $\Xi$ admits a factorization 
$$\Xi=\rho(\Xi)\,|\cdot|^b\left(\mu_{\Xi}\circ \NN_{K/\QQ}\right)$$ 
which is uniquely determined by the requirement that $\rho(\Xi)$ be unramified at $\fp^c$ and $\mu_\Xi$ be of finite order. Moreover, the $p$-adic avatar of $\rho(\Xi)$ necessarily factors through $\Omega_\p$.
\end{lemma}
\begin{proof}
Recall from Remark~\ref{remark_padic_Hecke_explicit} the characters $\Xi_\p$ and $\Xi_{\p^c}$, which we think of as Dirichlet characters of $p$-power conductor via the canonical isomorphisms
\begin{equation}
\label{eqn:identifyinglocalfactorsarpandpc}
\cO_{\p^c}^\times \stackrel{\sim}{\longleftarrow} \ZZ_p^\times\stackrel{\sim}{\lra} \cO_\p^\times\,.
\end{equation}
We take $\mu_\Xi:=\Xi_{\p^c}$ and $\rho(\Xi):=\Xi|\cdot|^{-b}\left(\mu_{\Xi}\circ \NN_{K/\QQ}\right)^{-1}$. Note that  $\rho(\Xi)$ is unramified at $\p^c$ and its $p$-adic avatar is of  the form
\begin{equation}\label{eqn_rho_exp_padic_avatar}
\widehat{\rho(\Xi)}(x)=u_{\fp}^{a-b} j(\Xi_{\p}\Xi_{\p^c}^{-1}(u_{\p}))
\end{equation}
under the notation of Remark~\ref{remark_padic_Hecke_explicit}. Since $\widehat{\Xi}$ and $\widehat{\mu_\Xi}\circ \NN_{K/\QQ}$ factor through $\Omega_\p\otimes\Omega_\cyc^\circ$, it follows that $\widehat{\rho(\Xi)}$ factors through  $\Omega_\p\otimes\Omega_\cyc^\circ$ as well. The explicit description of the $p$-adic avatar $\widehat{\rho(\Xi)}$ in \eqref{eqn_rho_exp_padic_avatar} shows in turn that $\widehat{\rho(\Xi)}$ factors through $\Omega_\p$, as required.
\end{proof}

\begin{remark}\label{rem_rhoXiforgeneralXi}
Suppose $\Xi$ is an arbitrary algebraic Hecke character of $\infty$-type $(a,b)$. The argument in the proof of Lemma~\ref{lem:rigidityofheckechars} in fact shows that $\Xi$ admits a factorization
$$\Xi=\rho(\Xi)\,|\cdot|^b\left(\mu_{\Xi}\circ \NN_{K/\QQ}\right)$$ 
which is uniquely determined by the requirement that $\rho(\Xi)$ be unramified at $\fp^c$ and $\mu_\Xi$ be of finite order. In the particular case when $\Xi=\chi\Psi$ where $\chi$ is as in the introduction and $\Psi \in \widetilde{\Sigma}$, one has  $\mu_{\Xi}=\mu_{\Psi}$ and $\rho(\Xi)=\chi\rho(\Psi)$.
\end{remark}
 
\subsection{$p$-adic Hecke characters of $K$} 

\begin{defn}
Let $\widetilde{\frak{Z}}$ denote the collection  of characters $\Xi$ on $\Omega_\p\times \Omega_\cyc^\circ$ which are of the form
$$\Xi(x)=\eta(x_\p)\langle x_\p \rangle^{a}\langle x_{\p^c}\rangle^{b} \cdot\nu\circ\NN_{K/\QQ}(x_{\rm fin}),$$ 
where  $\eta$ is a character of finite $p$-power order and $\p$-power conductor, $\nu$ is a Dirichlet character of $p$-power conductor, $a,b\in \Zp$ and $x_{\rm fin}\in \prod_{v\nmid \infty}(K\otimes \QQ_v)^\times$ is the non-archimedean component of the id\`ele $x$. The elements of $\widetilde{\frak{Z}}$ are called \emph{$p$-adic Hecke characters on $\Omega_\p\times \Omega_\cyc^\circ$}.  The pair $(a,b)$ is called the \emph{$p$-adic type} of $\Xi$.

Likewise, we let $\frak{Z}\subset \widetilde{\frak{Z}}$ denote the collection of $p$-adic Hecke characters factoring through $\Omega$.  Its elements have the form $\Xi(x)=\Xi_{\rm fin}(x_{\rm fin})\langle x_\p \rangle^{a}\langle x_{\p^c}\rangle^{b} $, where $\Xi_{\rm fin}$ is character of $p$-power order and of conductor dividing $p^\infty$. 
\end{defn}

Notice that if $\Xi \in \widetilde{\Sigma}$ (resp., $\Xi \in \Sigma$), then its $p$-adic avatar belongs to $\widetilde{\frak{Z}}$ (resp., to ${\frak{Z}}$). If the $\infty$-type of $\Xi$ is $(a,b)$, then the $p$-adic type of its $p$-adic avatar  also equals $(a,b)$. Furthermore, the $p$-adic avatars of Hecke characters $\Psi\in \Sigma_+$ considered as characters of $\Omega$ form a dense subset of the rigid analytic space of continuous $p$-adic characters of $\Omega$. We have the following $p$-adic analogue of Lemma~\ref{lem:rigidityofheckechars}.

\begin{lemma}
\label{lemma:factorpadiccharacters}
Let $\Xi \in \widetilde{\frak{Z}}$ be a $p$-adic Hecke character of $p$-adic type $(a,b)$. Then, there exists a uniquely determined factorization 
$$\Xi(x)=\rho(\Xi)(x)\cdot\NN_{K/\QQ}^b(x_{\rm fin})\cdot\left(\mu_\Xi\circ \NN_{K/\QQ}(x)\right),$$
where $\mu_\Xi$ is a finite Hecke character of $\QQ$, the $p$-adic Hecke character $\rho(\Xi)$ is unramified at $\p^c$ with $p$-adic type $(a-b,0)$. Moreover, the $p$-adic character $\rho(\Xi)$ necessarily factors through $\Omega_\fp$.
\end{lemma}
\begin{remark}
Lemma~\ref{lemma:factorpadiccharacters} allows us to identify $\widetilde{\frak{Z}}$ with the rigid analytic space $\textup{Sp} \,\LL^\dagger_L(\Omega_\p)\times \textup{Sp} \,\LL^\dagger_L(\Omega_\cyc^\circ)$.
\end{remark}
\begin{defn}
Let $\Xi$ be a Hecke character (algebraic or $p$-adic). We define the Hecke character dual to $\Xi$ by setting $\Xi^{D}:=(\Xi^{c})^{-1}$. If $\Xi$ has conductor $\ff(\Xi)$ and type $(a,b)$, then $\Xi^D$ has conductor $\frak{f}(\Xi)^c$ and type $(-b,-a)$.
\end{defn}
\begin{defn}
We let $\frak{Z}_{\ac}$ denote the space of anticyclotomic $p$-adic Hecke characters $($in more precise terms, it consists of those characters $\Psi\in \frak{Z}$ such that $\Psi^D=\Psi$$)$. 
\end{defn}

Note that we may identify $\frak{Z}_{\ac}$ with the rigid analytic space $\textup{Sp} \,\LL_{L}^\dagger(\Omega_\ac)$.
\subsection{Theta-series attached to algebraic Hecke characters of $K$}

\begin{defn}
\label{def:thetaseries}
\textup{(1)} Given an algebraic Hecke character $\Xi$ of $\infty$-type $(-u,0)$ with $u\geq 0$ and conductor $\ff(\Xi)$, we let 
$$\Theta(\Xi):=\sum_{(\frak{a},\frak{f}(\Xi))=1}\Xi(\frak{a})q^{\mathbb{N}_{K/\QQ}\frak{a}}$$ 
denote the associated theta-series.
\\
\textup{(2)} More generally, let $\Xi$ be an algebraic Hecke character of $\infty$-type $(a,b)$ with $a\le b$. Let $\eta_\Xi$ denote the unique Dirichlet character (of $\QQ$) of conductor $\mathbb{N}\ff(\rho(\Xi))$ which is characterized by the property that 
$$\eta_\Xi(n)=\rho(\Xi)((n))\cdot n^{a-b}$$
for all integers $n$ prime to $\mathbb{N}_{K/\QQ}\ff(\rho(\Xi))$. We write $\epsilon_K:=\left(\frac{D_K}{\cdot}\right)$ for the quadratic Dirichlet character attached to $K$. We set $N_\Xi:=D_K\cdot\mathbb{N}_{K/\QQ}\ff(\rho(\Xi))$ and $\theta_\Xi:=\eta_\Xi\epsilon_K$ (which is a Dirichlet character of conductor $N_\Xi$). Finally, we let 
$$g_\Xi=\Theta^{\textup{ord}}(\rho(\Xi)):=\sum_{(\frak{a},\frak{f}(\rho(\Xi))\p)=1}\rho(\Xi)(\frak{a})q^{{\NN_{K/\QQ}\frak{a}}}  \in S_{b-a+1}(\Gamma_1(N_\Xi),\theta_\Xi)$$
denote the associated $p$-ordinary theta series which is an eigenform of indicated weight, level and nebentype; it is a newform if and only if the conductor of $\rho(\Xi)$ is divisible by $\fp$. We also let $g^{[p]}_\Xi=\Theta^{[p]}(\rho(\Xi))$ denote its $p$-depletion and set 
$$g_\Xi^\circ=\Theta(\rho(\Xi)):=\sum_{(\frak{a},\frak{f}(\rho(\Xi)))=1}\rho(\Xi)(\frak{a})q^{{\NN_{K/\QQ}\frak{a}}}$$
so that $g_\Xi$ is the $p$-ordinary stabilization of the newform $g_\Xi^{\circ}$ whenever the conductor of $\rho(\Xi)$ is prime to $\fp$. 
\end{defn}

\begin{remark}
\label{rem:idetifysigmawithinthenearlyorddeformation}
Thanks to Lemma~\ref{lem:rigidityofheckechars}, choosing an algebraic Hecke character $\Xi\in \widetilde{\Sigma}_+$ (of $\infty$-type $(a,b)$, say) amounts to a choice of a triple $(\Theta(\chi\Xi_0),\eta,j)$, where $\Xi_0=\rho(\Xi)\in \Sigma_+$ is unramified at $\fp^c$ with $\infty$-type $(-u,0)$ where $u=b-a\geq 0$, $\eta=\mu_\Xi$ is a Dirichlet character of $\QQ$ of $p$-power conductor and $1\leq j=b+k/2 \leq k-1$ is an integer. 
\end{remark}
\begin{remark}\label{rem_demi_crys_calculations}
Suppose $\Xi=\chi\Psi$ is a Hecke character where $\Psi\in \widetilde{\Sigma}_+$ and $\chi$ is a ring class character modulo $\frak{f}$. As noted in Remark~\ref{rem_rhoXiforgeneralXi}, we have $\rho(\Xi)=\chi\rho(\Psi)$ and in turn (since $\chi$ is anticyclotomic, we have $\chi((n))=1$ for every non-zero $n\in \ZZ$), $\theta_\Xi=\theta_\Psi$. Moreover, $\Xi$ is demi-crystalline if and only if $\Psi$ is. The explicit description of the $p$-adic avatar of $\rho(\Psi)$ in \eqref{eqn_rho_exp_padic_avatar} tells us that when this is the case, $\eta_{\Psi}={\mathds{1}}$ is the trivial character and hence
$$\theta_\Xi=\theta_\Psi=\epsilon_K\,.$$
\end{remark}
The formalism in Definition~\ref{def:thetaseries} also applies to the dual Hecke character $\Xi^D$. As above, we may use Lemma~\ref{lem:rigidityofheckechars} to write $\Xi^D=\rho(\Xi^D)\,|\cdot|^{-a}(\mu_{\Xi^D}\circ \NN_{K/\QQ})$ (we caution the readers that $\rho(\Xi^D)\neq \rho(\Xi)^D$) and consider the associated $p$-depleted twisted theta-series 
\begin{align}
\label{eqn:maindualitycalculation}
\notag g^{[p]}_{\Xi^D}\otimes \mu_{\Xi^D}:=\Theta^{[p]}(\rho(\Xi^D)\cdot\mu_{\Xi^D}\circ\NN_{K/\QQ})
&=\sum_{(\frak{a},\frak{f}^cp)=1}\rho(\Xi^D)(\frak{a})\mu_{\Xi^D}(\NN_{K/\QQ}\frak{a})q^{\NN_{K/\QQ}\frak{a}}\\
\notag&=\sum_{(\frak{a},\frak{f}^cp)=1}\Xi^D(\frak{a})\NN_{K/\QQ}\frak{a}^{-a}q^{\NN_{K/\QQ}\frak{a}}\\
&=\sum_{(\frak{a},\frak{f}^cp)=1}\Xi(\frak{a}^{c})^{-1}\NN_{K/\QQ}\frak{a}^{-a}q^{\NN_{K/\QQ}\frak{a}}\\
\notag&=\sum_{(\frak{a},\frak{f}^cp)=1}\Xi(\frak{a}^{c})^{-1}(\NN_{K/\QQ}\frak{a}^c)^{-a}q^{\NN_{K/\QQ}(\frak{a}^c)}\\
\notag&=\sum_{(\frak{b},\frak{f}p)=1}\Xi(\frak{b})^{-1}\NN_{K/\QQ}(\frak{b})^{-a}q^{\NN_{K/\QQ}(\frak{b})}\\ 
\notag&=\overline{\Theta}^{[p]}(\rho(\Xi)\cdot\mu_\Xi\circ\NN_{K/\QQ})=\overline{g}^{[p]}_\Xi\otimes\mu_{\Xi}^{-1},
\end{align}
where the equality before the last one follows from the fact that $\overline{\Xi}(\frak{b})\Xi(\frak{b})=(\NN_{K/\QQ}\mathfrak{b})^{-a-b}$.
 \begin{defn}
We say that an algebraic Hecke character $\Xi$ of conductor dividing $\ff p^\infty$ is \emph{demi-crystalline} if the the character $\rho(\Xi)$ is crystalline. We say that it is \emph{crystalline} if it is demi-crystalline and $\mu_\Xi$ is trivial.
\end{defn}

\begin{remark}
\label{rem:crystallinemembersofthehidafamilyG}
Let $\rho(\Xi) \in{\Sigma}_{+}$ be an unramified algebraic Hecke character of $\infty$-type $(-u,0)$ with $u\equiv 0\mod p-1$ (whose $p$-adic avatar necessarily factors through $\Omega_\fp$). Then the $p$-stabilized eigenform $g_{\chi\rho(\Xi)} \in S_{u+1}(\Gamma_1(N_\chi p),\epsilon_K)$ is the unique crystalline weight $u+1$ specialization of the CM branch $\g$ of the Hida family interpolating $\{g_{\chi\Psi}\}_{\Psi\in {\Sigma}_{+}}$. It is not hard to see that all crystalline specializations arise in this manner. We further remark that the eigenform $g_{\chi\rho(\Xi)}$ is the $p$-stabilization of the newform $g_{\chi\rho(\Xi)}^\circ\in S_{u+1}(\Gamma_1(N_\chi),\epsilon_K) $.
\end{remark}
\begin{remark}
\label{rem:crysCMpointsduality}
Suppose in this remark that $\chi^D=\chi$. Let $\Xi$ be a demi-crystalline Hecke character of conductor dividing $\ff p^\infty$. The computation (\ref{eqn:maindualitycalculation}) above carries over for the theta-series $\Theta(\chi\rho(\Xi))$ and shows that 
$$g_{\chi\Xi^D}^\circ\otimes\mu_{\Xi^{D}}=\Theta(\chi\rho(\Xi^D))\otimes\mu_\Xi^{D}=\overline{\Theta}(\chi\rho(\Xi))=\overline{g}_{\chi\Xi}^\circ\otimes\mu_\Xi^{-1}\,.$$
Moreover,
\begin{equation}
\label{eqn:thedualityintwistedformunratp}
g_{\chi\Xi^D}^\circ=\overline{g}_{\chi\Xi}^\circ=g_{\chi\Xi}^\circ\otimes \theta^{-1}_{\chi\Xi,N_0},
\end{equation}
where $N_0$ is the prime-to-$p$ part of $N_{\chi\Xi}$ and $\theta_{\chi\Xi,N_{0}}$ denotes the prime-to-$p$ component of the nebentype $\theta_{\chi\Xi}$ (equivalently, it is the nebentype character for $g_{\chi\Xi}^\circ$). Indeed, we may verify the second (well-known) equality by comparing the corresponding Euler factors in the associated Hecke $L$-series of both sides. 
We also note that
\begin{equation}
\label{eqn:thedualityintwistedforforpstabilizedforms}
g_{\chi\Xi^D}=g_{\chi\Xi}\otimes \theta^{-1}_{\chi\Xi,N_0}
\end{equation}
for the $p$-stabilized theta series. Indeed, the fact that the Hecke eigenvalues of the eigenforms $\overline{g}_{\chi\Xi}$ and $g_{\chi\Xi}\otimes \theta^{-1}_{\chi\Xi,N_0}$ agree away from $p$ is well-known $($c.f., Remark 2.5 of \cite{loefflerERL}$)$. Moreover, a calculation similar to (\ref{eqn:maindualitycalculation}) shows that the same holds true for the pair $\overline{g}_{\chi\Xi}$ and ${g}_{\chi\Xi^D}$. It therefore remains to check that the $U_p$-eigenvalues acting on both sides are equal. This is verified through a direct computation:
\begin{align*}
\chi(\p^c)\rho(\Xi^D)(\p^c)&=\chi(\p^c)\cdot\rho(\Xi)^{-1}(\p)\cdot p^{-u}\\
&=\chi(\p^c)\cdot\rho(\Xi)(\p^c)\cdot\rho(\Xi)^{-1}((p))\cdot p^{-u}\\
&=\chi(\p^c)\cdot\rho(\Xi)(\p^c)\cdot \theta^{-1}_{\chi\Xi,N_0}(p),
\end{align*}
where the first equality is valid because $\rho(\Xi)$ is unramified at $\p$, so that $\rho(\Xi)^D|\cdot|^{-u}=\rho(\Xi^D)$.
\end{remark}

\subsection{Functional Equation for Rankin-Selberg $L$-functions} 
\label{subsec:functionalequationforcomplexLseries}
We will revisit the work of Li~\cite{WinnieLi79} and recast the functional equation she has established for the Rankin-Selberg $L$-functions in a form suitable for our purposes.

Fix an integer $\kappa\geq 2$ and a demi-crystalline algebraic Hecke character $\Psi \in \Sigma$ of $\infty$-type $(a,b)$ with 
$$1-k/2\leq a\leq b\leq \kappa-k/2-1$$ 
 and $\mu_\Psi\neq \mathbbm{1}$. Throughout, we let $h\in S_{\kappa}(\Gamma_1(N_f),\varepsilon_h)$ denote a newform (we will later choose the newform $h$ such that its $p$-stabilization is a member of a Coleman family through $f^\alpha$) and consider the newform 
 $$g=g_{\chi\Psi}\otimes \mu_\Psi^{-1}:=\Theta(\chi\rho(\Psi)\cdot\mu_\Psi^{-1}\circ \NN_{K/\QQ})=\Theta^{[p]}(\chi\rho(\Psi)\cdot\mu_\Psi^{-1}\circ \NN_{K/\QQ}) \in  S_{b-a+1}(\Gamma_1(N),\theta),$$ 
where $N=N_{\chi\Psi}$ and $\theta=\theta_{\chi\Psi}\mu_{\Psi}^{-2}$. Note that the equality follows from our assumption that $\mu_\Psi\neq\mathbb{1}$. As explained in \cite[Remark~2.2]{loefflerERL}, we have
$$L(h,g,s)=L^{\textup{imp}}(h,g_{\chi\Psi},\mu_\Psi^{-1},s)$$
for the motivic $L$-function $L(h,g,s)$ and the imprimitive Rankin-Selberg $L$-function $L^{\rm imp}(h,g_{\chi\Psi},\mu_\Psi^{-1},s)$ (see Definition 2.1 of op. cit. for the definition of the latter). Loeffler's formula~\cite[Prop. 2.12]{loefflerERL} allows us to interpolate $L$-values in this form. 
\begin{defn}\label{defn_root_number_good_cases}
We define the completed $L$-series $\LL_{h,g}(s)$ by setting 
$$\LL_{h,g}(s):=(2\pi)^{-2s+\kappa+b-a-1}\Gamma(s-\kappa+1)\Gamma(s)L^{\rm imp}(f,h,s)\,.$$
When $N_f$ and $N$ are coprime, we also define the global root number $\mathscr{W}(s)$ as
$$\mathscr{W}(s):=(N_fN)^{-2s+\kappa+b-a}\theta(-1)\varepsilon_h(N)\overline{\theta}(N_f)\lambda_{N_f}(h)^2\lambda_{N}(g)^2$$
where $\lambda_{N_f}(h)$ and $\lambda_{N}(g)$ are the Atkin-Lehner pseudo-eigenvalues.
\end{defn}
\begin{remark}
It is of course possible to define global root number when $N_f$ and $N$ have common prime divisiors. However, the explicit formula for the root number is far more complicated $($see (2.11) of \cite{WinnieLi79}$)$ and we shall be content to treat only the case when $N_f$ and $N$ are coprime. In particular, the conductor of $\chi$ is prime to $N_f$ and hence, whenever we would like to insist that $\epsilon(f/K,\eta)=\pm1$, we are forced to assume for the nebentype character is trivial. This is the only place where our assumption on the nebentype for $f$ plays an essential role. 
\end{remark}

The following is a restatement of Theorem~2.2 of \cite{WinnieLi79}, in view of Example 2 in \emph{op. cit}.
\begin{theorem}
\label{thm:complexfunctionalequation}
Suppose $(N_f,N)=1$. Then,
$$\LL_{h,g}(s)=\mathscr{W}(s)\LL_{\overline{h},\overline{g}}(\kappa+b-a-s).$$
\end{theorem}
In more explicit form, 
\begin{align*}
{(2\pi)^{-2s+\kappa+b-a-1}}\Gamma(s-\kappa+1)\Gamma(s)L^{\rm imp}(h,g,s)=&(N_fN)^{-2s+\kappa+b-a}\theta(-1)\varepsilon_h(N)\overline{\theta}(N_f)\lambda_{N_f}(h)^2\lambda_{N}(g)^2\\
&{\times(2\pi)^{2s-\kappa+a-b-1}}\Gamma(b-a+1-s)\Gamma(\kappa+b-a-s)\\
& \times L^{\rm imp}(\overline{h},\overline{g},\kappa+b-a-s)\,.
\end{align*}
We shall need this identity evaluated at $s=b+k/2$, which  reduces to
\begin{align*}
(-2\pi)^{\kappa-k-a-b}(a+k/2-1)!(b+k/2-1)!L^{\rm imp}(h,g,b+k/2)&=(N_fN)^{\kappa-k-a-b}\theta(-1)\varepsilon_h(N)\overline{\theta}(N_f)\\
&\times\lambda_{N_f}(h)^2\lambda_{N}(g)^2(2\pi)^{a+b-\kappa+k}\\
& \times (\kappa-k/2-a-1)!(\kappa-k/2-b-1)!\\
& \times L^{\rm imp}(\overline{h},\overline{g},\kappa-k/2-a)\,.
\end{align*}
Thence,
\begin{align}
\label{eqn:functionalequationtheratio}
\frac{L^{\rm imp}(h,g,b+k/2)}{L^{\rm imp}(\overline{h},\overline{g},\kappa-k/2-a)}=(-N_fN)^{\kappa-k-a-b}&\theta(-1)\varepsilon_h(N)\overline{\theta}(N_f)\lambda_{N_f}(h)^2\lambda_{N}(g)^2\\
&\notag\times(2\pi)^{2(a+b)-2(\kappa-k)}\frac{(\kappa-k/2-a-1)!(\kappa-k/2-b-1)!}{(a+k/2-1)!(b+k/2-1)!}\,.
\end{align}

\section{Functional equation for the three variable geometric $p$-adic $L$-function}
\label{sec:threevarfunctionaleqn}
Our goal of this section is to deduce a functional equation for the Beilinson-Flach $p$-adic $L$-function attached to $f^\alpha$. In order to do that, we need to deform $f^\alpha$ in a Coleman family of eigenforms and consider the restriction of a CM Hida family interpolating $\left\{g_{\chi\Psi}\right\}_{\Psi\in \Sigma_{+}}$ to an affinoid of the corresponding weight space. Throughout this section, we assume that $N_f$ and $pD_K\NN_{K/\QQ}(\ff)$ are coprime.
\begin{remark}
\label{rem:noeasywayout}
A more direct approach could have been possible (without the need to deform $f^\alpha$ to a Coleman family) whenever the following assertion holds:\\\\
Suppose for an element $\mathscr{G}\in\cH_{s_\alpha,s_\alpha}(\Gamma)$ we have $\mathscr{G}(\Psi)=0$ for every Hecke character $\Psi \in \Sigma_{+}(k)$. Then $\mathscr{G}$ is identically $0$.

When $s_\alpha<1$, this claim follows from the work of Loeffler $($see  \cite[Theorem~9]{LoefflerHeidelberg2012}$)$. Otherwise, we do not know how to verify this statement.
\end{remark}

Throughout this section, we assume that
\begin{equation}
\chi^D=\chi.\label{eq:chiD}
\end{equation}
Let $\f=\sum_{n=1}^{\infty}a_n(\f)q^n$ denote a Coleman family which is new of tame level $N_f$ and nebentype $\omega^k$. We suppose that $\f$ is defined over some affinoid neighborhood $\mathscr{X}$ in the weight space 
and its weight $k$ specialization is $f^\alpha$. Note that we retain our hypothesis that $\varepsilon_f=\mathbb{1}$, so that the prime-to-$p$ part of the nebentype of $\f$ is trivial and the conjugate Coleman family $\f^D$ (whose defining property is given as in Lemma 3.4 of \cite{loefflerERL}) coincides with $\f$. 

We identify the branch of the weight space corresponding to the Hida family $\g$ interpolating $\{g_{\chi\Psi}\}_{\Psi\in {\Sigma}_{+}}$ with ${\rm{Spf}}\, \LL_{L}(\Omega_\p)$ (c.f. the discussion on pages 2157-2158 in \cite{ghatevatsal2004AIF}). 
Notice that the tame level of $\g$ equals $N_\chi:=D_K\NN_{K/\QQ}\ff $ and its nebentype is $\theta_\chi\omega=\epsilon_K\omega$ thanks to \eqref{eq:chiD}, where $\omega$ is the $p$-adic Teichm\"uller character. 

Fix an affinoid $\mathscr{Y}\subset \textup{Sp}\,\LL_{L}^\dagger(\Omega_\p)$, which is stable under the map induced by $\gamma\mapsto \gamma^{-1}$ on $\Omega_\p$. We shall consider the restriction $\g\vert_\mathscr{Y}$ of the family $\g$ to $\mathscr{Y}$. 
\begin{defn}
\label{def:geometricpadicLfunction}
Let 
$$L_p^{\rm geom}(\f,\g\vert_\mathscr{Y})\in \cO(\mathscr{X})\widehat{\otimes}\cO(\mathscr{Y})\widehat{\otimes}\LL_{L}^\dagger(\Omega_\cyc^\circ)$$ 
denote the geometric $p$-adic $L$-function given as in \cite[Theorem~6.1]{loefflerERL}. 
\end{defn}
\begin{defn}
Set ${\widetilde{\frak{Z}}}(\mathscr{Y})=\mathscr{Y}\times \textup{Sp} \,\LL_{L}^\dagger(\Omega_\cyc^\circ) \subset\textup{Sp} \,\LL_{L}^\dagger(\Omega_\p)\times  \textup{Sp} \,\LL_{L}^\dagger(\Omega_\cyc^\circ)=:{\widetilde{\frak{Z}}}$. It consists of $p$-adic Hecke characters $\Xi\in \widetilde{\frak{Z}}$ whose factorization $\Xi(x)=\rho(\Xi)(x)\cdot x_\p^ax_{\p^c}^b\cdot\left(\mu_{\Xi}\circ \NN_{K/\QQ}x\right)$ given by Lemma~\ref{lemma:factorpadiccharacters} verifies $\rho(\Xi)\in \mathscr{Y}$. We define $\frak{Z}(\mathscr{Y})$ in a similar fashion.
\end{defn}

Given a $p$-adic Hecke character $\Psi\in \widetilde{\frak{Z}}(\mathscr{Y})$ and any specialization $\f(\kappa)$ of $\f$, we set 
$$L_p^{\rm geom}(\f,\g\vert_\mathscr{Y})\left(\f(\kappa),\Psi\right):=L_p^{\rm geom}(\f,\g\vert_\mathscr{Y})\left(\f(\kappa),g_{\chi\rho(\Psi)},b+k/2\right)$$
where $g_{\chi\rho(\Psi)}$ is the specialization of the family $\g$ corresponding to $\rho(\Psi) \in \mathscr{Y}$; it corresponds to a classical eigenform if and only if $\rho(\Psi)$ is the $p$-adic avatar of an algebraic Hecke character that belongs to $\Sigma_+$. This permits us to think of $L_p^{\rm geom}$ as an element of $ \cO(\mathscr{X})\,\widehat{\otimes}\,\cO(\widetilde{\frak{Z}}(\mathscr{Y}))$.

\begin{theorem}
\label{thm:mainfunctionaleqninthreevariables}
There exists an Iwasawa function $\mathscr{G}_f \in \LL_{\cO}(\Omega)$ so that 
$${L_p^{\rm geom}(\f,\g\vert_\mathscr{Y})(\f(\kappa),\Psi)}=\mathscr{G}_f(\Psi)\cdot \langle N_f^2 N_{\chi}\rangle^{\kappa-k}\cdot{L_p^{\rm geom}(\f,\g\vert_\mathscr{Y})(\f(\kappa),\Psi^D|\cdot|^{k-\kappa})}\,$$ 
for every $\Psi \in \widetilde{\frak{Z}}(\mathscr{Y})$ and $\kappa\in \mathscr{X}$. Moreover, when $\Psi\in \frak{Z}_\ac$ and $\f(\kappa)=f^\alpha$ (so that $\kappa=k$) we have $\mathscr{G}_f(\Psi)=\epsilon(f/K)\,.$
\end{theorem}

\begin{proof}
We set  
\begin{align*}
{\Sigma}_{\rm cris}:=\{(\f(\kappa),\Psi): \f(\kappa) \hbox{ is classical} &\hbox{ of}  \hbox{ weight } \kappa\in \ZZ_{\geq 2} \hbox{ and } \Psi\in \Sigma \hbox{ is demi-crystalline with } \\
&\infty\hbox{-type }   (a,b),\, 1-k/2\leq a\leq b\leq \kappa- k/2-1  \hbox{ and } \mu_\Psi\neq \mathbbm{1}\}\subset \mathscr{X}\times \widetilde{\frak{Z}}(\mathscr{Y})
\end{align*} 
and observe that $(\f(\kappa),\Psi^D|\cdot|^{k-\kappa})$ belongs to ${\Sigma}_{\rm cris}$ whenever $(\f(\kappa),\Psi)$ does: Indeed, Remark~\ref{rem:crysCMpointsduality} shows (relying on the fact that $\rho(\Psi)$ is unramified) that $\rho(\Psi^D)=\rho(\Psi)^D|\cdot|^{b-a}$ is unramified and that $\mu_{\Psi^D}=\mu_\Psi^{-1}$. We let $\f(\kappa)^\circ$ denote the newform of level $N_f$ such that $\f(\kappa)$ is a $p$-stabilization of $\f(\kappa)^\circ$.

For $(\f(\kappa),\Psi)\in {\Sigma}_{\rm cris}$, we shall compute the ratio
\begin{equation}\label{eqn:trivialsimplification01}
\frac{L_p^{\rm{geom}}(\f,\g\vert_\mathscr{Y})(\f(\kappa),\Psi)}{L_p^{\rm{geom}}(\f,\g\vert_\mathscr{Y})(\f(\kappa),\Psi^D|\cdot|^{k-\kappa})}
\end{equation}
and prove that it interpolates in a desired manner,  as $(\f(\kappa),\Psi)$ over the dense subset ${\Sigma}_{\rm cris}\subset  \mathscr{X}\times \widetilde{\frak{Z}}(\mathscr{Y})$. In order to achieve that, we will make use of Loeffler's interpolation formula~\cite[Proposition 2.10 and Theorem 6.3]{loefflerERL}. 

It follows from the said interpolation formula that 
\begin{align*}
L_p^{\rm{geom}}(\f,\g\vert_\mathscr{Y})(\f(\kappa),\Psi)=\frac{\mathcal{E}(\f(\kappa),\Psi)}{\mathcal{E}(\f(\kappa))\,\mathcal{E}^*(\f(\kappa))}&\times\frac{(b+k/2-1)!(a+k/2-1)!\,i^{\kappa-b+a-1}}{\langle \f(\kappa)^\circ,\f(\kappa)^\circ\rangle_{N_f}\pi^{a+b+k}2^{a+b+k+\kappa-1}}\\
&\times L^{\rm{imp}}({\f}(\kappa)^\circ,g_{\chi\Psi}^\circ\otimes\mu_\Psi^{-1},b+k/2),
\end{align*}
(where $\mathcal{E}(\f(\kappa),\Psi)$ corresponds to the factor denoted by $\mathcal{E}(\f(\kappa),g_{\chi\Psi},b+k/2+\mu_\Psi)$ in op. cit. and we shall provide its explicit expression below) and that
\begin{align*}
L_p^{\rm{geom}}(\f,\g\vert_\mathscr{Y})(\f(\kappa),\Psi^D|\cdot|^{k-\kappa})=\frac{\mathcal{E}(\f(\kappa),\Psi^D|\cdot|^{k-\kappa})}{\mathcal{E}(\f(\kappa))\,\mathcal{E}^*(\f(\kappa))}&\times\frac{(\kappa-k/2-a-1)!(\kappa-k/2-b-1)!\,i^{\kappa-b+a-1}}{{\langle \f(\kappa)^\circ,\f(\kappa)^\circ\rangle_{N_f}}\pi^{2\kappa-a-b-k}2^{3\kappa-k-b-a-1}}\\
&\times L^{\rm{imp}}({\f}(\kappa)^\circ,g_{\chi\Psi^D}^\circ\otimes\mu_\Psi,\kappa-k/2-a),
\end{align*}
where $\mathcal{E}(\f(\kappa),\Psi^D|\cdot|^{k-\kappa})$ equals Loeffler's $\mathcal{E}(\f(\kappa),g_{\chi\Psi^D},\kappa-a-k/2-\mu_\Psi)$. We therefore have
\begin{align}
\label{eqn:theratioofpadicL01}
\frac{L_p^{\rm{geom}}(\f,\g\vert_\mathscr{Y})(\f(\kappa),\Psi)}{L_p^{\rm{geom}}(\f,\g\vert_\mathscr{Y})(\f(\kappa),\Psi^D|\cdot|^{k-\kappa})}=\frac{\mathcal{E}(\f(\kappa),\Psi)}{\mathcal{E}(\f(\kappa),\Psi^D|\cdot|^{k-\kappa})}&\times \frac{(b+k/2-1)!(a+k/2-1)!}{(\kappa-k/2-a-1)!(\kappa-k/2-b-1)!}\\
\notag&\times(2\pi)^{2(\kappa-k)-2(a+b)}\frac{L({\f}(\kappa)^\circ,g,b+k/2)}{L({\f}(\kappa)^\circ,\overline{g},\kappa-k/2-a)},
\end{align}
where we have set $g:=g_{\chi\Psi}^\circ\otimes\mu_\Psi^{-1}\in S_{b-a+1}(\Gamma_1(N_{\chi\Psi}),\theta_{\chi\Psi}\mu_\Psi^{-2})$ to be the newform we considered in Section~\ref{subsec:functionalequationforcomplexLseries} and we have used Remark~\ref{rem:crysCMpointsduality} to identify $\overline{g}=\overline{g}_{\chi\Psi}^\circ\otimes\mu_\Psi$ with $g_{\chi\Psi^D}^\circ\otimes\mu_\Psi$.  Let $p^r>1$ denote the conductor of $\mu_\Psi$ so that $N_{\chi\Psi}=N_\chi p^{2r}$. 
We recall that $\varepsilon_f=\mathbbm{1}=\chi\chi^c$ and $\f(\kappa)^\circ=\overline{\f(\kappa)}^\circ$ in our set up (since we assumed that $(N_{\chi},N_f)=1$ in order to make use of Theorem~\ref{thm:complexfunctionalequation}).  Combining (\ref{eqn:functionalequationtheratio}) and (\ref{eqn:theratioofpadicL01}) along with the fact that $\theta_{\chi\Psi}(-1)=\epsilon_K(-1)=-1$ (as we have observed in Remark~\ref{rem_demi_crys_calculations}), 
we conclude that
\begin{align}
\label{eqn:theratioofpadicLsimplified}
\notag\frac{L_p^{\rm{geom}}(\f,\g\vert_\mathscr{Y})(\f(\kappa),\Psi)}{L_p^{\rm{geom}}(\f,\g\vert_\mathscr{Y})(\f(\kappa),\Psi^D|\cdot|^{k-\kappa})}= 
{-}&\,\frac{\mathcal{E}(\f(\kappa),\Psi)}{\mathcal{E}(\f(\kappa),\Psi^D|\cdot|^{k-\kappa})}\times (-N_fN_{\chi}p^{2r})^{\kappa-k-a-b}\,\theta_{\chi\Psi}^{-1}\,\mu_\Psi^{2}(N_f)\,\lambda_{N_f}(\f(\kappa)^\circ)^2\lambda_{N_{\chi\Psi}}(g)^2\\
\notag=&\, {-}({-}N_fN_{\chi})^{\kappa-k-a-b}\,\frac{G(\mu_\Psi)^2}{G(\overline{\mu}_\Psi)^2}\,\theta_{\chi\Psi}^{-1}\,\mu_\Psi^{2}(N_f)\,\lambda_{N_{\chi\Psi}}(g)^2\lambda_{N_f}(\f(\kappa)^\circ)^2\\
=&{-}\,\left({-}N_fN_{\chi}\right)^{-a-b}\,\frac{G(\mu_\Psi)^2}{G(\overline{\mu}_\Psi)^2}\,\theta_{\chi\Psi}^{-1}\,\mu_\Psi^{2}(N_f)\lambda_{N_f}(f)^2\lambda_{N_{\chi\Psi}}(g)^2\cdot\langle N_f^2N_{\chi}\rangle^{\kappa-k}\\
\notag=&:\,\mathscr{G}_f(\Psi)\cdot \langle N_f^2N_{\chi}\rangle^{\kappa-k},
\end{align}
where we have used 
\begin{align*}
\mathcal{E}(\f(\kappa),\Psi)&=G(\mu_\Psi)^2\left(\frac{p^{2b+k-2}}{\alpha(\kappa)^2p^{b-a}}\right)^{r}=G(\mu_\Psi)^2\left(\frac{p^{a+b+k-2}}{\alpha(\kappa)^2}\right)^{r}\\
\mathcal{E}(\f(\kappa),\Psi^{D}|\cdot|^{k-\kappa})&=G(\overline{\mu}_\Psi)^2\left(\frac{p^{2\kappa-2a-k-2}}{\alpha(\kappa)^2p^{b-a}}\right)^{r}=G(\overline{\mu}_\Psi)^2\left(\frac{p^{2\kappa-k-a-b-2}}{\alpha(\kappa)^2}\right)^{r}
\end{align*}
 for the second equality (here, $\alpha(\kappa) \in \cO(\mathscr{X})^\times$ denotes the eigenvalue for the $U_p$-action on $\f(\kappa)$);  the third equality holds because $\kappa-k\equiv 0 \mod p-1$, $\f(\kappa)$ is crystalline and $\lambda_{N_f}(\f(\kappa)^\circ)^2=\langle N_f\rangle^{\kappa}$; the final equality is our definition of the $p$-adic integer $\mathscr{G}_f(\Psi)$.

Recall that $g$ is given as the Rankin-Selberg convolution $g_{\chi\Psi}^\circ\otimes\mu_\Psi^{-1}$. Since $\Psi$ is demi-crystalline, the character $\rho(\Psi)$ is crystalline by definition. In particular, the theta-series $g_{\chi\Psi}^\circ$ is of level $N_\chi$, which is prime to $p$. 
By \cite[page 228]{atkinli}, we have 
\begin{align*}
\lambda_{N_{\chi\Psi}}(g)=\lambda_{N_\chi}(g^\circ_{\chi\Psi})\mu_{\Psi}^{-1}(-N_\chi)\theta_{\chi\Psi}(p^r)\frac{G(\mu_\Psi^{-1})}{G(\mu_\Psi)}.
\end{align*}
Therefore,
\begin{align*}
\mathscr{G}_f(\Psi)&={-}({-}N_fN_\chi)^{-a-b}\theta_{\chi\Psi}(p^{2r}/N_f)\mu_\Psi^2(N_f/N_\chi)\lambda_{N_f}(f)^2\lambda_{N_\chi}(g_{\chi\Psi}^\circ)^2\\
&={\left(\frac{D_K}{-N_f}\right)}({-}N_fN_\chi)^{-a-b}\mu_\Psi^2(N_f/N_\chi)\lambda_{N_f}(f)^2\lambda_{N_\chi}(g_{\chi\Psi}^\circ)^2,
\end{align*}
where the second equality follows recalling our assumption that $\Psi$ be demi-crystalline and {noting that 
$$\theta_{\chi\Psi}(p^{2r}/N_f)=\epsilon_K(p^{2r})\epsilon_K(N_f)=\left(\frac{D_K}{N_f}\right)$$ 
by our discussion in Remark~\ref{rem_demi_crys_calculations}}. Each term in the final expression for $\mathscr{G}_f(\Psi)$ varies analytically in $\Psi$ and interpolates to an Iwasawa function since $N_fN_\chi$ is prime to $p$ and $\lambda_{N_\chi}(g_{\chi\Psi}^\circ)$ is given by $i^{a-b-1}W(\chi\rho(\Psi))/N_\chi^{1/2}$, where $W(\chi\rho(\Psi))$ denotes the Gauss sum for $\chi\rho(\Psi)$ (see \cite[(3.3.8)]{miyake} for its precise definition). 

When $\Psi$ is anticyclotomic, so is $\chi\Psi$. In this case, $a+b=0$ and
$$\mathscr{G}_f(\Psi)=\left(\frac{D_K}{-N_f}\right)\mu_\Psi^2(N_f/N_\chi)\lambda_{N_f}(f)^2\lambda_{N_\chi}(g_{\chi\Psi}^\circ)^2\,.$$ 
{On comparing this expression with the formula for $\mathscr{W}(k/2)$ given in Definition~\ref{defn_root_number_good_cases} (with $h=f$, $\varepsilon_h=\mathds{1}$, $g=g_{\chi\Psi}^\circ\otimes\mu_{\Psi}^{-1}$ and $\theta=\theta_{\chi\Psi}\mu_\Psi^{-2}=\epsilon_K\mu_\Psi^{-2}$), we conclude that $\mathscr{G}_f(\Psi)=\epsilon(f/K\otimes \chi\Psi)$, the global root number of $L(f/K\otimes \chi\Psi,s)$ at $s=k/2$}. {In particular, we have $\mathscr{G}_f(\Psi)=\epsilon(f/K)$ for all anticyclotomic characters $\Psi \in \Sigma_{\rm cris}\cap \Sigma(k)$}.
\end{proof}
\begin{remark}
\label{rem:thereisnothingweird}
This remark is for readers who might feel uneasy about the power of the norm character on the right side of the functional equation. The main reason for its presence is the fact that $s=k/2$ is not the center of the functional equation for the Rankin-Selberg $L$-series $L(\Bf(\kappa)/K,\Psi,s)$ in general, where $\Bf(\kappa)$ is a specialization of the Coleman family $\Bf$ of weight $\kappa \in \ZZ_{\geq 2}$ and $\Psi$ is a Hecke character as before.

We elaborate regarding this point. Let $\widetilde{L}(\Bf(\kappa)/K,\Psi,s)$ denote $L(\Bf(\kappa)/K,\Psi,s+k/2)$. The functional equation for the Rankin-Selberg $L$-series reads
\begin{equation}
\label{eqn:fnuctionalequationabridgedversion}
\widetilde{L}(\Bf(\kappa)/K,\Psi,s)\dot{=}\widetilde{L}(\Bf(\kappa)/K,\Psi^D,\kappa-k-s)
\end{equation}
where $\dot{=}$ means equality up to simple fudge factors. The value of the geometric $p$-adic $L$-function $L_p^{\rm geom}(\f,\g\vert_\mathscr{Y})$ at the pair $(\Bf(\kappa),\Psi)$ equals $\widetilde{L}(\Bf(\kappa)/K,\Psi,0)$ up to some interpolation factors, whereas its value at the pair $(\Bf(\kappa),\Psi^D|\cdot|^{k-\kappa})$ equals $($still up to the same interpolation factors$)$
$$\widetilde{L}(\Bf(\kappa)/K,\Psi^{D}|\cdot|^{k-\kappa},0)=\widetilde{L}(\Bf(\kappa)/K,\Psi^{D},\kappa-k).$$ The functional equation (\ref{eqn:fnuctionalequationabridgedversion}) in turn relates $L_p^{\rm geom}(\f,\g\vert_\mathscr{Y})(\Bf(\kappa),\Psi)$ to $L_p^{\rm geom}(\f,\g\vert_\mathscr{Y})(\Bf(\kappa),\Psi^D|\cdot|^{k-\kappa})$ and it is precisely this relation that we interpolate above as $\kappa$ and $\Psi$ vary.
\end{remark}

For $\lambda,\mu\in\{\alpha,\beta\}$, we recall from \cite[\S3.4]{BFsuper} that there exists a two-variable $p$-adic $L$-function
\[
\frak{L}_{\lambda,\mu}\in\cH_{s_\lambda,s_\mu}(\Gamma),
\]
given by the local image of the $\mu$-stabilized Beilinson-Flach elements over $K$ under the two-variable Perrin-Riou map attached to the $\lambda^{-1}$-eigenspace of the Dieudonn\'e module of $f$ at $p$. When $\lambda=\mu=\alpha$, this $p$-adic $L$-function coincides with the geometric $p$-adic $L$-function $L_p^{\rm geom}(\f,\g\vert_\mathscr{Y})$ we studied above with $\f$ specialized at $f^\alpha$ and $\g$ is any CM family over $K$.

\begin{defn}
  Let $\tau: \Gamma \ra \Gamma$ denote the involution induced by $\gamma_{\p}\mapsto \gamma_{\p^c}^{-1}$ and $\gamma_{\p^c}\mapsto \gamma_{\p}^{-1}$. This in turn induces an involution 
$$\tau:  \cH_{s_\lambda,s_\mu}(\Gamma) \lra  \cH_{s_\mu,s_\lambda}(\Gamma)\,.$$
The image of $H$ will be denoted by $H^\tau$. 
\end{defn}

\begin{theorem}
\label{thm:functinaleqn2var}
Let $\mathscr{G}_f\in \LL_{\cO}(\Gamma)$ be as in the statement of Theorem~\ref{thm:mainfunctionaleqninthreevariables}. Then 
$$\frak{L}_{\alpha,\alpha}=\mathscr{G}_f\cdot \frak{L}_{\alpha,\alpha}^\tau.$$
Moreover, the projection $\mathscr{G}_f^\ac \in \LL_{\cO}(\Gamma_\ac)$ identically equals $\epsilon(f/K)$.
\end{theorem}
The identical statement for the other root $\beta$ of the Hecke polynomial is valid as well, as the proof below does not distinguish between the two roots. 
\begin{proof}
Given an arbitary demi-crystalline algebraic Hecke character $\Psi$, choose a large enough affinoid neighborhood $\mathscr{Y}\subset \textup{Sp}\,\LL_{L}^\dagger(\Omega_\p)$ (since $\g$ is ordinary, we may choose $\mathscr{Y}$ as large as we like) in a way that we may compute 
\begin{align*}
\left(\frak{L}_{\alpha,\alpha}-\mathscr{G}_f\cdot \frak{L}_{\alpha,\alpha}^\tau\right)(\Psi)&=\frak{L}_{\alpha,\alpha}(\Psi)-\mathscr{G}_f(\Psi)\cdot \frak{L}_{\alpha,\alpha}(\Psi^D)\\
&=L_p^{\rm{geom}}(\f,\g\vert_\mathscr{Y})(f^\alpha,\Psi)-\mathscr{G}_f(\Psi)L_p^{\rm{geom}}(\f,\g\vert_\mathscr{Y})(f^\alpha,\Psi^D)=0\,.
\end{align*}
Our assertion follows from the density of the demi-crystalline points in $\textup{Sp}\,\LL_{L}^\dagger(\Omega_\p)\,\widehat{\otimes}\,\textup{Sp}\,\LL_{L}^\dagger(\Omega_\cyc^\circ)$. 
\end{proof}

\begin{corollary}
\label{cor:vanishingoftthefullanticyclotomicpadicLfunction}
If $\epsilon(f/K)=-1$, then $\frak{L}_{\alpha,\alpha}^\ac=0$.
\end{corollary}

\section{The vanishing of anticyclotomic doubly signed $p$-adic $L$-functions}
\label{subsec:vanishingofdoublysignedpadicL}

In this section, we assume that $\epsilon(f/K)=-1$ and $a_p(f)=0$ (so that $\beta=-\alpha$).  Let $L_0$ be the sub-extension of $L$ generated by  the images of all Fourier coefficients of $f$ under $\iota_p$ as well as all the values of $\chi$. 
We recall from \cite[\S3.4]{BFsuper} that the  $p$-adic $L$-functions $\frak{L}_{\lambda,\mu}$ can be decomposed in to doubly-signed $p$-adic $L$-functions. More precisely, let $\frak{L}_{\pm,\pm}$  be the  $p$-adic $L$-functions denoted by $\frak{L}_{\bullet,\star}$ with $\bullet,\star\in\{\#,\flat\}$ in op. cit. They are  the images of the signed Beilinson-Flach elements (given by (20) in op. cit.) under the signed Coleman maps at $\p$ (defined in \S2.5 of op. cit.). Note that both the signed Beilinson-Flach elements and the signed Coleman maps are defined over $L_0$, which implies that $\frak{L}_{\pm,\pm}\in \Lambda_{L_0}(\Gamma)$.   Equation (22) in op. cit. says that
\begin{align}
\notag
\begin{pmatrix}
\frak{L}_{\alpha,\alpha} &\frak{L}_{\alpha,-\alpha}\\
\frak{L}_{-\alpha,\alpha} &\frak{L}_{-\alpha,-\alpha}\
\end{pmatrix}=&\begin{pmatrix}
-\Tw_{k/2-1}\tlog_{k-1,\fp^c}^-/2& \Tw_{k/2-1}\tlog_{k-1,\fp^c}^+/2\alpha\\
 \Tw_{k/2-1}\tlog_{k-1,\fp^c}^-/2& \Tw_{k/2-1}\tlog_{k-1,\fp^c}^+/2\alpha
\end{pmatrix}\times\begin{pmatrix}
\frak{L}_{-,-} &\frak{L}_{+,-}\\
\frak{L}_{-,+} &\frak{L}_{+,+}\
\end{pmatrix}\\
&\times \begin{pmatrix}
-\Tw_{k/2-1}\tlog_{k-1,\fp}^-/2&  \Tw_{k/2-1}\tlog_{k-1,\fp}^-/2\\
 \Tw_{k/2-1}\tlog_{k-1,\fp}^+/2\alpha& \Tw_{k/2-1}\tlog_{k-1,\fp}^+/2\alpha
\end{pmatrix},\label{eqn:factorizationap0}
\end{align} 
for some  half-logarithmic functions $\Tw_{k/2-1}\tlog_{k-1,\fp}^\pm$ and $\Tw_{k/2-1}\tlog_{k-1,\fp^c}^\pm$, which are defined as follows.

Let  $\Brig$ be the set of power series in $\Qp[[\pi]]$ that converge on the open unit $p$-adic disc, equipped with a $\Qp$-linear action $\vp:\pi\mapsto (1+\pi)^p-1$ and an action of $\Gamma_\cyc$ by $\gamma\cdot \pi \mapsto(1+\pi)^{\chi_\cyc(\gamma)}-1$, where $\chi_\cyc$ is the cyclotomic character on $\Gamma_\cyc\rightarrow 1+p\Zp$. The $\Lambda_{\Zp}(\Gamma_\cyc)$-linear map $\Lambda_{\Zp}(\Gamma_\cyc)\rightarrow \Zp[[\pi]]$ defined by $1\mapsto 1+\pi$ gives an isomorphism $\fm:\Lambda_{\Zp}(\Gamma_\cyc)\stackrel{\sim}{\longrightarrow}(1+\pi)\vp(\Zp[[\pi]])$ for some left inverse $\psi$ of $\vp$. This extends to $\fm:\Lambda_{\Qp}^\dagger(\Gamma_\cyc)\stackrel{\sim}{\longrightarrow}(1+\pi)\vp(\Brig)$. Define 
 \begin{align*}
\tlog_{k-1}^+&=\fm^{-1}\left((1+\pi)\left(\frac{1}{p}\prod_{n\ge 1}^\infty\frac{\vp^{2n}( q)}{p\vp^{2n+1}(\delta)}\right)^{k-1}\right),\\
\tlog_{k-1}^-&=\fm^{-1}\left((1+\pi)\left(\frac{1}{p}\prod_{n=1}^\infty\frac{\vp^{2n+1}( q)}{p\vp^{2n}(\delta)}\right)^{k-1}\right),
\end{align*} 
where $q=\vp(\pi)/\pi$ and $\delta=p/(q-\pi^{p-1})\in\Zp[[\pi]]^\times$. Note that $\tlog_{k-1}^\pm\in\cH_{(k-1)/2}(\Gamma_\cyc)$. Let $j\in \ZZ$ and $r\in \mathbb{R}_{\ge0}$. We define $\Tw_j:\cH_r(\Gamma_\cyc)\rightarrow \cH_r(\Gamma_\cyc)$ to be the twisting map induced by $\sigma\mapsto \chi_{\cyc}(\sigma)^j\sigma$, where $\sigma\in\Gamma_{\cyc}$ is considered as a group-like element in $\cH_r(\Gamma_\cyc)$.

For $\q\in\{\p,\p^c\}$, the half-logarithmic functions  $\Tw_{k/2-1}\tlog_{k-1,\q}^\pm$ are defined to be the functions obtained from $\Tw_{k/2-1}\tlog_{k-1}^\pm$ on replacing  $\gamma_\cyc$ by  $\gamma_\q$.
 
\begin{theorem}
\label{thm:L5}
If  the ramification index of $L_0/\QQ_p$ is odd, then $\frak{L}_{+,+}^\ac=\frak{L}_{-,-}^\ac=0$. 
\end{theorem}
\begin{remark}
In other words, if $a_p(f)=0$ and the ramification index of $L_0/\QQ_p$ is odd then the property $\mathbf{(L5)}$ of \cite{BFsuper} holds true. 
\end{remark}

 We recall Pollack's plus/minus logarithms from \cite{pollack03}:
 \begin{align}
 \log_{k-1}^+&=\prod_{j=0}^{k-2}\frac{1}{p}\prod_{n=1}^\infty\frac{\Phi_{2n}(u^{-j}\gamma_\cyc)}{p},\label{eq:pollack1}\\
 \log_{k-1}^-&=\prod_{j=0}^{k-2}\frac{1}{p}\prod_{n=1}^\infty\frac{\Phi_{2n-1}(u^{-j}\gamma_\cyc)}{p}\label{eq:pollack2},
 \end{align} 
 where $\Phi_m$ denotes the $p^m$-th cyclotomic polynomial for $m\in\ZZ_{\ge1}$ and $u=\chi_\cyc(\gamma_\cyc)$.
   Via the isomorphism given in \cite[(2.2)]{LLZ3}, up to multiplication by a unit in $\Lambda_{\Zp}(\Gamma_\cyc)^\times$, the half-logarithmic functions $\log_{k-1}^\pm$ agree with $\tlog_{k-1}^\pm$.

 We now study functional equations for these half-logarithms.
 \begin{lemma}\label{lem:pollackFE}
 There exist units $u^\pm\in \Lambda_{\Zp}(\Gamma_\cyc)^\times$ such that
 \[
 \left(\Tw_{k/2-1}\log_{k-1}^\pm\right)^\tau=u^{\pm}\Tw_{k/2-1}\log_{k-1}^\pm.
 \]
 \end{lemma}
 \begin{proof}
 For all $m\ge1$ and $j\in\ZZ$, we have
 \begin{align*}
  \Phi_{m}(u^{-j}\gamma_\cyc)^\tau&=\left(\frac{(u^{-j}\gamma_\cyc)^{p^m}-1}{(u^{-j}\gamma_\cyc)^{p^{m-1}}-1}\right)^\tau\\
& =\frac{(u^{-j}\gamma_\cyc^{-1})^{p^m}-1}{(u^{-j}\gamma_\cyc^{-1})^{p^{m-1}}-1}\\
&=(u^{j}\gamma_\cyc)^{p^{m-1}-p^m}\Phi_{m}(u^{j}\gamma_\cyc).
 \end{align*}
Since $u^{j}\gamma_\cyc$ is a unit in $\Lambda_{\Zp}(\Gamma_\cyc)$, our result follows.
 \end{proof}
 \begin{remark}\label{rk:units}
 We have the following explicit expressions for $u^\pm$:
 \begin{align*}
 u^+&=\prod_{j=1-k/2}^{k/2-1}\prod_{n=1}^\infty(u^{-j}\gamma_\cyc)^{p^{2n-1}-p^{2n}},\\
 u^-&=\prod_{j=1-k/2}^{k/2-1}\prod_{n=1}^\infty(u^{-j}\gamma_\cyc)^{p^{2n-2}-p^{2n-1}}.
 \end{align*}
 Since $u\in 1+p\Zp$, we infer that $u^\pm\in 1+p\Zp+(\gamma_\cyc-1)\Lambda_{\Zp}(\Gamma_\cyc)$.
  \end{remark}
 
 \begin{corollary}\label{cor:tildeFE}
There exist units $\tilde u^\pm\in \Lambda_{\Zp}(\Gamma_\cyc)^\times$ such that
 \[
 \left(\Tw_{k/2-1}\tlog_{k-1}^\pm\right)^\tau=\tilde{u}^{\pm}\Tw_{k/2-1}\tlog_{k-1}^\pm.
 \]
 Furthermore, $\tilde u^\pm\in 1+p\Zp+(\gamma_0-1)\Lambda_{\Zp}(\Gamma_\cyc)$.
\end{corollary}
\begin{proof}
Suppose that $\Tw_{k/2-1}\tlog_{k-1}^\pm=v^\pm\Tw_{k/2-1}\log_{k-1}^\pm$, where $v^\pm\in\Lambda_{\Zp}(\Gamma_\cyc)$. Then, Lemma~\ref{lem:pollackFE} tells us that
\[
\left(\Tw_{k/2-1}\tlog_{k-1}^\pm\right)^\tau=\frac{(v^\pm)^\tau}{v^\pm} u^\pm\left(\Tw_{k/2-1}\tlog_{k-1}^\pm\right).
\]
Since $v^\pm\in \Lambda_{\Zp}(\Gamma_\cyc)^\times$, the quotient $(v^\pm)^\tau/v^\pm\in 1+(\gamma_\cyc-1)\Lambda_{\Zp}(\Gamma_\cyc)$. Hence we are done by Remark~\ref{rk:units}.
\end{proof}

\begin{remark}\label{rk:FElog}
For $\q\in\{\p,\p^c\}$, let $\tilde u_\q^\pm$ denote the element obtained from $\tilde u^\pm$ on replacing $\gamma_0$ by $\gamma_\q$, where $\tilde u^\pm$ are given in Corollary~\ref{cor:tildeFE}. We may translate the functional equation given in the statement of Corollary~\ref{cor:tildeFE} to:
\[
\left(\Tw_{k/2-1}\tlog_{k-1,\q^c}^\pm\right)^\tau=\tilde{u}_\q^{\pm}\Tw_{k/2-1}\tlog_{k-1,\q}^\pm.
\]
\end{remark}

\begin{proof}[Proof of Theorem~\ref{thm:L5}]
Our strategy here is inspired by an argument due to Castella and Wan in \cite{castellawan1607}, where the doubly-signed $p$-adic $L$-functions of an elliptic curve are studied. The factorisation formula \eqref{eqn:factorizationap0} for $\mathfrak{L}_{\alpha,\alpha}$ gives
\begin{align*}
4\alpha^2\fL_{\alpha,\alpha}=&\Tw_{k/2-1}\tlog_{k-1,\p}^+\Tw_{k/2-1}\tlog_{k-1,\p^c}^+\fL_{+,+}-\alpha\Tw_{k/2-1}\tlog_{k-1,\p}^-\Tw_{k/2-1}\tlog_{k-1,\p^c}^+\fL_{-,+}\\
&-\alpha\Tw_{k/2-1}\tlog_{k-1,\p}^+\Tw_{k/2-1}\tlog_{k-1,\p^c}^-\fL_{+,-}+\alpha^2\Tw_{k/2-1}\tlog_{k-1,\p}^-\Tw_{k/2-1}\tlog_{k-1,\p^c}^-\fL_{-,-}\,.
\end{align*}
On combining the functional equations from Theorem~\ref{thm:functinaleqn2var} and Remark~\ref{rk:FElog}, we deduce that the quantity
\begin{equation}\label{eq:L0}
\begin{split}
\Tw_{k/2-1}\tlog_{k-1,\p}^+\Tw_{k/2-1}\tlog_{k-1,\p^c}^+\left(\fL_{+,+}-\tilde u^+_{\fp}\tilde u^+_{\fp^c}\mathscr{G}_f\fL_{+,+}^\tau\right)
\\
+\alpha^2\Tw_{k/2-1}\tlog_{k-1,\p}^-\Tw_{k/2-1}\tlog_{k-1,\p^c}^-&\left(\fL_{-,-}-\tilde u^-_{\fp}\tilde u^-_{\fp^c}\mathscr{G}_f\fL_{-,-}^\tau\right)
\end{split}
\end{equation}
is equal to
\begin{equation}\label{eq:Lalpha}
\begin{split}
\alpha\Tw_{k/2-1}\tlog_{k-1,\p}^-\Tw_{k/2-1}\tlog_{k-1,\p^c}^+\left(\fL_{-,+}-\tilde u^-_{\fp}\tilde u^+_{\fp^c}\mathscr{G}_f\fL_{+,-}^\tau\right)
\\
+\alpha \Tw_{k/2-1}\tlog_{k-1,\p}^+\Tw_{k/2-1}\tlog_{k-1,\p^c}^-&\left(\fL_{+,-}-\tilde u^+_{\fp}\tilde u^-_{\fp^c}\mathscr{G}_f\fL_{-,+}^\tau\right).
\end{split}
\end{equation}
{Recall that $v_p$ is the $p$-adic valuation on $L$ normalized by $v_p(p)=1$. In particular, $v_p(\alpha)=\frac{k-1}{2}\in \frac{1}{2}+\ZZ$ as $k$ is even. Let $e$ be the ramification index of $L_0/\Qp$. Since $\fL_{\pm,\pm}$,  $\Tw_{k/2-1}\tlog_{k-1,\p}^\pm$ and $\Tw_{k/2-1}\tlog_{k-1,\p^c}^\pm$ are all defined over $L_0$, as a power series in $L[[\gamma_{\p}-1,\gamma_{\p^c}-1]]$,  the non-zero coefficients in \eqref{eq:L0} have valuations in $\frac{1}{e}\ZZ$, whereas those in \eqref{eq:Lalpha} have valuations in $v_p(\alpha)+\frac{1}{e}\ZZ=\frac{1}{2}+\frac{1}{e}\ZZ$. Therefore, under our hypothesis that $e$ is odd, both \eqref{eq:L0} and \eqref{eq:Lalpha} have to be zero.} In particular,
\begin{align}\Tw_{k/2-1}\tlog_{k-1,\p}^+\Tw_{k/2-1}\tlog_{k-1,\p^c}^+&\left(\fL_{+,+}-\tilde u^+_{\fp}\tilde u^+_{\fp^c}\mathscr{G}_f\fL_{+,+}^\tau\right)=\notag\\
&\label{eq:FEpm}-\alpha^2\Tw_{k/2-1}\tlog_{k-1,\p}^-\Tw_{k/2-1}\tlog_{k-1,\p^c}^-\left(\fL_{-,-}-\tilde u^-_{\fp}\tilde u^-_{\fp^c}\mathscr{G}_f\fL_{-,-}^\tau\right)\,.
\end{align}

Let $\pi_\ac$ be the projection to the anticyclotomic line (parallel to the cyclotomic line). If  $\theta$ is a character on $\Gamma_\ac$ which sends $\gamma_\ac$ to a primitive $p^{2n-1}$-st root of unity, then \eqref{eq:pollack1} and \eqref{eq:pollack2} tell us that
\begin{align*}
\pi_\ac\left(\Tw_{k/2-1}\tlog_{k-1,\p}^-\Tw_{k/2-1}\tlog_{k-1,\p^c}^-\right)(\theta)=0;\\
\pi_\ac\left(\Tw_{k/2-1}\tlog_{k-1,\p}^+\Tw_{k/2-1}\tlog_{k-1,\p^c}^+\right)(\theta)\ne0.
\end{align*}
Therefore, on applying $\pi_\ac$ to \eqref{eq:FEpm}, we deduce that $$\fL_{+,+}^\ac-(\tilde u^+_{\fp}\tilde u^+_{\fp^c})^\ac\mathscr{G}^\ac_f\fL_{+,+}^\ac\in  \LL_{L}(\Gamma_\ac)$$ vanishes at infinitely many characters of finite order. Thus, it must be identically $0$. Given that $\mathscr{G}^\ac_f=\epsilon(f/K)=-1$ by Theorem~\ref{thm:mainfunctionaleqninthreevariables}, we conclude that
\[
\fL_{+,+}^\ac+(\tilde u^+_{\fp}\tilde u^+_{\fp^c})^\ac\fL_{+,+}^\ac=0.
\]
The last part of Corollary~\ref{cor:tildeFE} tells us that $1+(\tilde u^+_{\fp}\tilde u^+_{\fp^c})^\ac\ne0$ and this shows that $\fL_{+,+}^\ac=0$. The proof for $\fL_{-,-}^\ac=0$ is similar.
\end{proof}

\begin{remark}
It is not clear to us how to deduce the vanishing of $\fL_{+,-}^\ac$ or of $\fL_{-,+}^\ac$ based on the argument we present above. This is the reason why we have restricted our attention to the symmetric choice of signs.
\end{remark}

\section*{Acknowledgments}
We would like to thank David Loeffler for answering our questions regarding his work in \cite{loefflerERL}. We would also like to thank the anonymous referee for carefully reading an earlier version of the manuscript and their helpful suggestions and comments, which led to many improvements of the presentation of the paper.
 
\bibliographystyle{amsalpha}
\bibliography{references}

\providecommand{\bysame}{\leavevmode\hbox to3em{\hrulefill}\thinspace}
\providecommand{\MR}{\relax\ifhmode\unskip\space\fi MR }
\providecommand{\MRhref}[2]{%
  \href{http://www.ams.org/mathscinet-getitem?mr=#1}{#2}
}
\providecommand{\href}[2]{#2}
\begin{thebibliography}{LLZ17}

\bibitem[AL78]{atkinli}
A.~O.~L. Atkin and Wen Ch'ing~Winnie Li, \emph{Twists of newforms and
  pseudo-eigenvalues of {$W$}-operators}, Invent. Math. \textbf{48} (1978),
  no.~3, 221--243.

\bibitem[BL16]{BFsuper}
K{\^a}z{\i}m B\"uy\"ukboduk and Antonio Lei, \emph{{I}wasawa theory of elliptic
  modular forms over imaginary quadratic fields at non-ordinary primes}, 2016,
  preprint, arXiv:arXiv:1605.05310.

\bibitem[CW16]{castellawan1607}
Francesc Castella and Xin Wan, \emph{Perrin-{R}iou's main conjecture for
  elliptic curves at supersingular primes}, arXiv:1607.02019, 2016.

\bibitem[GV04]{ghatevatsal2004AIF}
Eknath Ghate and Vinayak Vatsal, \emph{On the local behaviour of ordinary
  {$\Lambda$}-adic representations}, Ann. Inst. Fourier (Grenoble) \textbf{54}
  (2004), no.~7, 2143--2162 (2005).

\bibitem[Li79]{WinnieLi79}
Wen Ch'ing~Winnie Li, \emph{{$L$}-series of {R}ankin type and their functional
  equations}, Math. Ann. \textbf{244} (1979), no.~2, 135--166.

\bibitem[LLZ17]{LLZ3}
Antonio Lei, David Loeffler, and Sarah~Livia Zerbes, \emph{On the asymptotic
  growth of {B}loch-{K}ato-{S}hafarevich-{T}ate groups of modular forms over
  cyclotomic extensions}, Canad. J. Math. \textbf{69} (2017), no.~4, 826--850.

\bibitem[Loe14]{LoefflerHeidelberg2012}
David Loeffler, \emph{{$p$}-adic integration on ray class groups and
  non-ordinary {$p$}-adic {$L$}-functions}, Iwasawa theory 2012, Contrib. Math.
  Comput. Sci., vol.~7, Springer, Heidelberg, 2014, pp.~357--378.

\bibitem[Loe18]{loefflerERL}
\bysame, \emph{A note on {$p$}-adic {R}ankin-{S}elberg {$L$}-functions}, Canad.
  Math. Bull. \textbf{61} (2018), no.~3, 608--621.

\bibitem[LZ16]{LZ1}
David Loeffler and Sarah~Livia Zerbes, \emph{Rankin-{E}isenstein classes in
  {C}oleman families}, Res. Math. Sci. \textbf{3} (2016), Paper No. 29, 53.

\bibitem[Miy89]{miyake}
Toshitsune Miyake, \emph{Modular forms}, Springer-Verlag, Berlin, 1989,
  Translated from the Japanese by Yoshitaka Maeda.

\bibitem[Pol03]{pollack03}
Robert Pollack, \emph{On the {$p$}-adic {$L$}-function of a modular form at a
  supersingular prime}, Duke Math. J. \textbf{118} (2003), no.~3, 523--558.

\end{thebibliography}

\end{document}